\documentclass{amsart}

\usepackage{amsfonts,amsmath,amssymb,amsthm,amscd,amsxtra}
\usepackage{enumerate,epsfig,verbatim}
\usepackage{mathptmx}
\usepackage{amsfonts,amsmath,amssymb,amsthm,amscd,amsxtra}
\usepackage{enumerate,verbatim}
\usepackage[all,2cell,ps]{xy}
\usepackage[notcite,notref]{}
\usepackage[pagebackref]{hyperref}
\usepackage[all,2cell,ps]{xy}
\usepackage{amsfonts}
\usepackage[margin=1.4in]{geometry}
\usepackage{color}
\usepackage[notcite,notref]{}
\usepackage{url}
\usepackage{datetime}
\usepackage{mathtools}

\usepackage[all,2cell,ps]{xy}
\usepackage{amssymb}
\usepackage{amsmath}
\usepackage{amsfonts}
\usepackage{amsmath}
\usepackage{amsthm}
\usepackage{amssymb}
\usepackage{amscd}
\usepackage{amsfonts}
\usepackage{amsxtra}     
\usepackage{epsfig}
\usepackage{verbatim}
\usepackage[all]{xypic}

\usepackage{amsmath}
\usepackage{amsthm}
\usepackage{amssymb}
\usepackage{amscd}
\usepackage{amsxtra}     
\usepackage{epsfig}
\usepackage{verbatim}
\usepackage[all]{xypic}
\usepackage{enumerate}

\SelectTips{cm}{}

\def\latex/{{\protect\LaTeX}}
\def\latexe/{{\protect\LaTeXe}}
\def\amslatex/{{\protect\AmS-\protect\LaTeX}}
\def\tex/{{\protect\TeX}}
\def\amstex/{{\protect\AmS-\protect\TeX}}
\def\bibtex/{{Bib\protect\TeX}}
\def\makeindx/{\textit{MakeIndex}}

\theoremstyle{plain} 

\newtheorem{thm}{Theorem}[section]

\newtheorem{prop}[thm]{Proposition}
\newtheorem{cor}[thm]{Corollary}

\theoremstyle{definition}

\newtheorem{obs}[thm]{Observation}

\newtheorem{eg}[thm]{Example}

\newtheorem{ques}[thm]{Question}

\newtheorem{rmk}[thm]{Remark}

\newcommand{\fp}{\mathfrak{p}}
\newcommand{\fq}{\mathfrak{q}}

 \DeclareMathOperator{\Tor}{Tor}
\DeclareMathOperator{\Ext}{Ext}
\DeclareMathOperator{\Hom}{Hom}

 \DeclareMathOperator{\Ass}{Ass}
 
 \DeclareMathOperator{\Supp}{Supp}
 \DeclareMathOperator{\Spec}{Spec}

  \DeclareMathOperator{\rank}{rank}
 \DeclareMathOperator{\pd}{pd}
 \DeclareMathOperator{\height}{height}

 \DeclareMathOperator{\depth}{depth}

 \DeclareMathOperator{\md}{\operatorname{\mathsf{mod}}}
 \def\Tr{\mathsf{Tr}\hspace{0.01in}}
 
\newcommand{\grade}{\operatorname{grade}}
 
 \newcommand{\Min}{\textup{Min}}

\def\urltilda{\kern -.15em\lower .7ex\hbox{\~{}}\kern
  .04em}\def\urldot{\kern -.10em.\kern -.10em}\def\urlhttp{http\kern
  -.10em\lower -.1ex\hbox{:}\kern -.12em\lower 0ex\hbox{/}\kern
  -.18em\lower 0ex\hbox{/}} 

\newcommand{\bb}{\left[ \begin{smallmatrix}}
\newcommand{\eb}{\end{smallmatrix} \right]}

\newcounter{eqc} 
  {\end{list}}%

\newcounter{prt}
  {\end{list}}%

\newcounter{rqm}
  {\end{list}}%

  {\end{list}}%

\begin{document}

\title[On the second rigidity theorem of Huneke and Wiegand]{On the second rigidity theorem of Huneke and Wiegand}

\author[O. Celikbas]{Olgur Celikbas}
\address{Olgur Celikbas\\
Department of Mathematics \\
West Virginia University\\
Morgantown, WV 26506-6310, U.S.A}
\email{olgur.celikbas@math.wvu.edu}

\author[R. Takahashi]{Ryo Takahashi}
\address{Ryo Takahashi\\ Graduate School of Mathematics, Nagoya University, Furocho, Chikusaku, Nagoya 464-8602, Japan and Department of Mathematics, University of Kansas, Snow Hall, 1460 Jayhawk Blvd, Lawrence, KS 66045-7523, USA}
\email{takahashi@math.nagoya-u.ac.jp}
\urladdr{http://www.math.nagoya-u.ac.jp/~takahashi/}

\thanks{Takahashi was partially supported by JSPS Grant-in-Aid for Scientific
Research 16K05098 and 16KK0099.}

\maketitle{}

\vspace{-0.2in}
\begin{center}
\small{\emph{Dedicated to the memory of Ragnar-Olaf Buchweitz}}
\end{center}

\begin{abstract}
In 2007 Huneke and Wiegand announced in an erratum 
that one of the conclusions of their depth formula theorem is flawed 
due to an incorrect convention for the depth of the zero module. 
Since then, the deleted claim has remained unresolved. In this paper 
we give examples to prove that the deleted claim is false, in general. 
Moreover, we point out several places in the literature which relied 
upon this deleted claim or the initial argument from 2007.

\end{abstract}

\section{Introduction}

In the following, unless otherwise stated, $R$ denotes a commutative Noetherian local ring and $\md R$ denotes the category of all finitely generated $R$-modules. For the standard notations and unexplained terminology, we refer the reader to \cite{EG} and \cite{thebook}.

The aim of this paper is to establish a result that yields examples of modules $M, N\in \md R$ such that $M$ and $M\otimes_RN$ are both reflexive, $N$ is not reflexive, and $N$ has finite projective dimension (in particular $N$ has \emph{rank}); see Theorem \ref{t1}, and Examples \ref{ornek1} and \ref{ornek2}. The motivation of this investigation stems from beautiful results of Huneke and Wiegand, namely from the following theorems:

\begin{thm} (Huneke and Wiegand \cite[2.5]{HW1}) \label{HWt1} Let $R$ be a complete intersection and let $M,N\in \md R$ be nonzero modules. If $\Tor_i^R(M,N)=0$ for all $i\geq 1$, then the following equality holds: $$\depth_R(M)+\depth_R(N)=\depth(R)+\depth_R(M\otimes_RN).$$
\end{thm}

The depth equality in Theorem \ref{HWt1} is generally referred to as the \emph{depth formula}; it was initially proved by Auslander when one of the modules considered has finite projective dimension; see \cite[1.2]{Au}. In codimension one case, Huneke and Wiegand established a condition on the tensor product $M\otimes_RN$ that yields Tor-independence; their result is called the \emph{second rigidity theorem} and is given as follows.

\begin{thm} (Huneke and Wiegand \cite[2.7]{HW1}) \label{HWt2} Let $R$ be a hypersurface and let $M,N\in \md R$, either of which has (constant) rank. If $M\otimes_RN$ is reflexive, then $\Tor_i^R(M,N)=0$ for all $i\geq 1$. 
\end{thm}

In passing, it seems worth noting an easy, albeit an important, consequence of the second rigidity theorem: over a hypersurface ring that has finite Cohen-Macaulay representation type, the depth of tensor products of two (nonzero) maximal Cohen-Macaulay modules cannot exceed one; see \cite[1.3]{bounds}.

Recall a module $M\in \md R$ satisfies $(S_n)$ if $\depth_{R_{\fp}}(M_{\fp}) \geq \min\{n, \dim(R_{\fp})\}$ for each $\fp \in \Supp_R(M)$; see \cite[Chapter 0]{EG} and \cite[page 451]{HW1}. Note that, if $R$ is Gorenstein, then $M$ is reflexive if and only if $M$ satisfies $(S_2)$ \cite[3.6]{EG}. Next is the question we are mainly concerned with in this paper; both parts of the question are true in some special cases, for example, if $R$ is a domain \cite[1.3]{CG}. 

\begin{ques} \label{sorun} Let $R$ be a complete intersection ring of codimension $c$, and let $M,N\in \md R$ be nonzero modules such that $M\otimes_RN$ satisfies $(S_n)$ for some positive integer $n$.
\begin{enumerate}[\rm(i)]
\item If $\Tor_i^R(M,N)=0$ for all $i\geq 1$, then must both $M$ \emph{and} $N$ satisfy $(S_n)$? 
\item If $M$ or $N$ has rank, $c=1$ and $n=2$, then must both $M$ \emph{and} $N$ be reflexive?
\end{enumerate}
\end{ques}

Although we record it here as a question, initially, the first part of Question \ref{sorun} was stated as a corollary of Theorem \ref{HWt1} in \cite{HW1}; see \cite[2.6]{HW1}. Similarly, the second part of the question was indeed part of the second rigidity theorem proved in \cite{HW1}; see \cite[2.7]{HW1}. In 2007 Huneke and Wiegand announced in the erratum \cite{erratum} that both of these results (i.e., both parts of Question \ref{sorun}) need to be removed from \cite{HW1}; this is because, in \cite{HW1}, contrary to the correct depth convention $\depth(0)=\infty$, it is assumed that $\depth(0)=-1$. As mentioned in \cite{erratum}, the depth lemma may fail in case one uses the convention $\depth(0)=-1$. It was also explained in \cite{erratum} that the proofs of Theorems \ref{HWt1} and \ref{HWt2} are intact under the assumption $\depth(0)=\infty$, but the claimed conclusions of these theorems, namely those stated as Question \ref{sorun}, do not follow from Theorems \ref{HWt1} and \ref{HWt2}. Moreover, it was not discussed in the erratum \cite{erratum} whether or not these removed results are false in general, or whether or not they may be justified via different techniques. In other words, Question \ref{sorun} has been open until now; see \cite[page 111]{GORS2}.

There are straightforward cases where both parts of Question \ref{sorun} are correct. For example, if both modules considered have full support, e.g., if the ring is a domain, then the question is positive: in this case, one can localize the depth formula at a prime ideal, obtain nonzero modules and follow the argument of \cite[2.8]{ArY}; see also \cite[1.3]{CG}. However, it turned out to be quite difficult to study Question \ref{sorun} to prove affirmative results. For example, Celikbas and Piepmeyer \cite{CG} attacked the problem by using a version of the new intersection theorem, and obtained partial results over complete intersection rings. 

In Section 2, we prove our main result that gives negative answers to Question \ref{sorun}; see Theorem \ref{t1}, and Examples \ref{ornek1}, and \ref{ornek2}. Along the way we make a new observation on Tor-rigidity, a topic initiated by Auslander \cite{Au}, but not well-understood in commutative algebra. In view of the negative answers we obtained for Question \ref{sorun}, some of the results from the literature, besides those in \cite{HW1}, need revisions; for example, see \cite[2.8]{ArY}, \cite[second and the third paragraphs on page 685]{HJW} and \cite[1.6(1)]{HW2}.



\section{Main result and examples}

In the following, $\Tr(M)$ denotes the \emph{Auslander transpose} of $M$ over $R$ \cite{AuBr}. Also we set $\depth(0)=\infty$. We start by recalling a property that will be often used tacitly; see, for example, \cite[3.4, 3.5 and 3.6]{EG}.

\begin{rmk} Let $R$ be a Noetherian ring (not necessarily local) and let $M\in \md R$ be a module. 
\begin{enumerate}[\rm(i)]
\item Suppose $R$ satisfies $(S_1)$. Then $M$ satisfies $(S_1)$ if and
only if $M$ is torsion-free.
\item Suppose $R$ is Gorenstein. Then $M$ satisfies 
$(S_1)$ if and only if $M$ is torsionless.
\item Suppose $R$ is Gorenstein. Then $M$ satisfies
$(S_2)$ if and only if $M$ is reflexive.
\end{enumerate}
\end{rmk}

Torsionless modules are torsion-free, but the converse is not true,  in general. For example, if $k$ is a field, $R=k[\![x, y, z]\!]/(x^2, xy, y^2)$ and $M=R/(x)$, then $R$ is a one-dimensional Cohen-Macaulay local ring, and $M$ is a torsion-free $R$-module since $\Ass_R(M)=\Min_R(M)=\{(x,y)\}$. However, $M$ is not torsionless since $\Ext_R^1(\Tr M,R)=\Ext_R^1(R/(x),R)$ is not zero: this can be checked by definition, or by using the exact sequence $0 \to R/(x,y) \to R \to R/(x) \to 0$. Vascencelos \cite[Theorem A.1]{Vas} proved that ``torsion-free'' and ``torsionless'' are equivalent notations if $R_{\fp}$ is Gorenstein for all associated primes $\fp$ of $R$.

Our aim is to establish negative answers to Question \ref{sorun}. However, let us first prove a special affirmative result. More precisely, we obtain a positive answer to the second part of Question \ref{sorun} in case both modules considered have depth two; see Corollary \ref{corpositive}. It seems that this result may be useful in further studying the torsion properties of tensor products of modules.

\begin{prop} \label{Prop2nd} Let $R$ be a local hypersurface and $M,N \in \md R$. Assume $\Tor_i^R(M,N)=0$ for all $i\geq 1$. Assume further, for some positive integer $n$, $M$ satisfies $(S_{n})$ and $\depth(M)=n$  (e.g., $M$ is locally free on the punctured spectrum of $R$ and $\depth(M)=n$). If $M\otimes_RN$ satisfies $(S_{n})$, then $N$ satisfies $(S_{n})$.
\end{prop}

\begin{proof} Let $\fp \in \Supp(N)$. If $p\in \Supp(M)$, then using Theorem \ref{HWt1} and localizing the depth formula, we see $N$ satisfies $\depth(N_\fp) \ge \min\{n, \dim(R_{\fp})\}$. Hence we may assume $\fp\notin \Supp(M)$. We know $M$ or $N$ has finite projective dimension; see \cite[1.9]{HW2}. Since $\fp \notin \Supp(M)$, it follows that $\Supp(M) \neq \Spec(R)$. Therefore $\pd(M)=\infty$: as otherwise, $M$ would have positive rank and so full support; see \cite[1.3]{CG}. Thus $\pd(N)<\infty$. Now Theorem \ref{HWt1} (see also \cite[1.2]{Au}) implies $\pd(N)=\depth(M)-\depth(M\otimes_RN)$. Since $\depth(M\otimes_RN)\geq n$, we have $\pd(N)\leq \depth(M)-n=0$, i.e., $N$ is free.
\end{proof}

\begin{cor} \label{corpositive} Let $R$ be a local hypersurface and let $M,N \in \md R$ be modules, either of which has rank. Assume $M\otimes_RN$ is reflexive. If $\depth(M)=\depth(N)=2$, then $M$ and $N$ are both reflexive.
\end{cor}

\begin{proof} It follows from Theorem \ref{HWt2} that $\Tor_i^R(M,N)=0$ for all $i\geq 1$. Moroever, we know $M$ or $N$ is reflexive; see \cite[4.4]{CG}. In either case, since both $M$ and $N$ have depth two, we conclude from Proposition \ref{Prop2nd} that both $M$ and $N$ are reflexive.
\end{proof}

Next is the statement of our main result. 

\begin{thm} \label{t1} Let $R$ be a commutative Noetherian ring (not necessarily local) and let $\fp\in \Spec(R)$ with $\grade(\fp)\geq 1$ and $\height(\fp)=n+1$ for some nonnegative integer $n$. Let $X\in \md R$ be a module and set $N=\Tr(X)$. Assume the following conditions hold:
\begin{enumerate}[\rm(i)]
\item $R$ satisfies $(S_{n+1})$.
\item $X$ is torsion, $X_{\fp}$ is not free, but $X_{\fq}$ is free for each $\fq \in \Spec(R)$  with $\fq \nsupseteq \fp$.
\item There exists a module $0\neq M\in \md R$ such that $M$ satisfies $(S_{n+2})$ and $\fp \notin \Supp_R(M)$.
\end{enumerate}
Then the following hold:
\begin{enumerate}[\rm(1)]
\item $M\otimes_RN$ satisfies $(S_{n+1})$. 
\item $\Tor_i^R(M,N)=0$ for all $i\ge 1$, and $\pd_R(N)=1$.
\item $N$ satisfies $(S_{n})$, but $N$ does not satisfy $(S_{n+1})$.
\end{enumerate}
\end{thm}

Here are some examples that give negative answers to Question \ref{sorun}. In the first example $M$ is a maximal Cohen-Macaulay module, but it is not in the second one. Recall that Question \ref{sorun} has a positive answer whenever the ring $R$ is a domain.

\begin{eg} \label{ornek1} Let $k$ be a field, $R=k[\![x,y,z,w]\!]/(xy)$, $\fp=(y,z,w)$, $N=\Tr(R/\fp)$ and let $M=R/(x)$. Then $R$ is a three-dimensional hypersurface and $M$ is a maximal Cohen-Macaulay $R$-module. In particular $M$ satisfies $(S_r)$ for each $r\geq 0$. Note also that $\height(\fp)=2$. Thus, with $n=1$, the hypotheses of Theorem \ref{t1} are satisfied. Therefore $\pd_R(N)=1$, $\Tor_i^R(M,N)=0$ for all $i\ge 1$, $M\otimes_{R}N$ and $M$ are both reflexive, $N$ is torsion-free, but $N$ is not reflexive. 
\pushQED{\qed} 
\qedhere
\popQED
\end{eg}

\begin{eg} \label{ornek2} Let $k$ be a field, $R=k[\![x,y,z,w,u]\!]/(xy)$, $\fp=(x,z, w)$, $N=\Tr(R/\fp)$ and let $M=\Omega^3_{R/(y)}(k)$, the third syzygy of $k$ over $R/(y)$. Then $R$ is a four-dimensional hypersurface, $\height(\fp)=2$, and $R/(y)$ is a maximal Cohen-Macaulay $R$-module.

There is an exact sequence $0 \to M \to (R/(y))^{\oplus a} \to (R/(y))^{\oplus b} \to (R/(y))^{\oplus c} \to k \to 0$  in $\md R$. Hence $\depth_R(M)=3$ (in particular $M$ is not a maximal Cohen-Macaulay) and $\fp \notin  \Supp_R(M)$. Since $R/(y)$ satisfies $(S_3)$ as an $R$-module, so does $M$. Thus, with $n=1$, the hypotheses of Theorem \ref{t1} are satisfied. Therefore $\pd_R(N)=1$, $\Tor_i^R(M,N)=0$ for all $i\ge 1$, $M\otimes_{R}N$ and $M$ are both reflexive, $N$ is torsion-free, but $N$ is not reflexive.
\pushQED{\qed} 
\qedhere
\popQED
\end{eg}

Prior to giving our proof of Theorem \ref{t1}, we record a result motivated by the arguments of Dutta \cite {Dutta}. For completeness, we include an elementary argument for our observation; see \cite[2.4 and 2.5]{Dutta} for a more general result.

\begin{obs} \label{obs1} Let $R$ be a commutative Noetherian ring (not necessarily local) and
let $M, N\in \md R$ be modules such that $\pd_R(N)\leq 1$ and $M$ is torsion-free. Then
$\Tor_i^R(M,N)=0$ for all $i\geq 1$.

In fact, set $r=\rank(N)$. Then there exists a short exact sequence $0
\to R^{\oplus r} \to N \to C \to 0$, where $C$ is torsion. Since
$\pd_R(C)\leq 1$ and $\grade_R(C)\ge1$, we see it suffices to assume
$\grade_R(N)\ge1$. In that case, pick a non zero-divisor $x$ on $R$
such that $xN=0$. Then $x$ is a non zero-divisor on $M$. As
$\pd_R(N)\le1$, the exact sequence $0\to M\xrightarrow{x}M\to
M/xM\to0$ induces an exact sequence $0 \to \Tor_1^R(M,N)
\stackrel{x}{\rightarrow}\Tor_1^R(M,N)$, and $x\Tor_1^R(M,N) =0$. This
yields $\Tor_1^R(M,N)=0$.
\pushQED{\qed} 
\qedhere
\popQED
\end{obs}

We are now ready to prove Theorem \ref{t1}.

\begin{proof}[Proof of Theorem \ref{t1}] Note, as $\Hom_R(X,R)=0$, there is an exact sequence $0 \to  F  \to G \to N \to 0$, where $F$ and $G$ are free $R$-modules. Note also that, since $X$ is nonzero and torsion, it cannot be projective. Hence we see $\pd_R(N)=1$. As $M$ is torsion-free, it follows from Observation \ref{obs1} that $\Tor_i^R(M,N)=0$ for all $i\ge 1$.

Recall that $X_{\fp}$ is not free, but $X_{\fq}$ is free for each $\fq \in \Spec(R)$  with $\fp \nsubseteq \fq$. Hence, for a given $\fq \in \Spec(R)$, $N_{\fq}$ is free if $\fp \nsubseteq \fq$, and $\pd_{R_{\fq}}(N_{\fq})=1$ if $\fp \subseteq \fq$. In particular, $\pd_{R_{\fp}}(N_{\fp})=1$.

Moreover, since $R$ satisfies $(S_{n+1})$ and $\height(\fp)=n+1$, we have that $R_{\fp}$ is Cohen-Macaulay. Therefore $\depth_{R_{\fp}}(N_{\fp})=\depth(R_{\fp})-1=(n+1)-1=n < n+1=\min\{n+1, \height(\fp)\}$. This shows that $N$ does not satisfy $(S_{n+1})$.

Let $\fq \in \Supp_R(N)$ such that $\fp \subseteq \fq$. Then we have:
\begin{equation}\tag{\ref{t1}.3}
\depth_{R_{\fq}}(N_{\fq})=\depth(R_{\fq})-\pd_{R_{\fq}}(N_{\fq})\geq \min\{n+1, \height(\fq)\}-1= (n+1)-1=n.
\end{equation}
If $\fp\nsubseteq \fq$, then $X_{\fq}$ is free by assumption, and so is $N_{\fq}$. This fact and the equation in (\ref{t1}.3) shows that $N$ satisfies $(S_n)$.

Now let $\fq \in \Supp(M\otimes_RN)$. We will prove $\depth_{R_{\fq}}(M\otimes_RN)_{\fq} \geq 
\min\{n+1, \height(\fq) \}$. If $\fp \nsubseteq \fq$, then $\depth_{R_{\fq}}(M\otimes_RN)_{\fq} = \depth_{R_{\fq}}(M_{\fq})\geq 
\min\{n+2, \height(\fq) \}$ since $N_{\fq}$ is free and $M$ satisfies $(S_{n+2})$. Hence  we may assume $\fp \subseteq \fq$. Notice $\fp \neq \fq$ since $ \fp \notin \Supp_R(M)$. In particular, we have $\height(\fq)\geq n+2$. Since $\pd_{R_{\fq}}(N_{\fq})=1$, and $M_{\fq}\neq 0 \neq N_{\fq}$, \cite[1.2]{Au} yields:
\begin{align}\notag{}
\depth_{R_{\fq}}(M_{\fq}\otimes_{R_{\fq}}N_{\fq})  & =  \depth_{R_{\fq}}(M_{\fq})-1 \notag  \geq  \min \{n+2, \height(\fq)\}-1=(n+2)-1=n+1. 
\end{align}
This establishes that  $M\otimes_RN$ satisfies $(S_{n+1})$, and hence completes the proof.
\renewcommand{\qedsymbol}{$\square$}
\end{proof}

\begin{rmk} Assume $R$, in Theorem \ref{t1}, is local and $\Tr$ is defined using minimal free resolutions. Then, if one picks $X=R/\fp$, it follows that the rank of $N$ is one less than the cardinality of a minimal generating set of the prime ideal $\fp$; see the proof of Theorem \ref{t1}, and Examples \ref{ornek1} and \ref{ornek2}. \pushQED{\qed} 
\qedhere
\popQED
\end{rmk}

We finish this section with a few remarks on Tor-rigidity. Recall that a module $M \in \md R$ is called \emph{Tor-rigid} if, whenever $\Tor_n^R(M,N)=0$ for some $N\in \md R$ and some $n\geq 0$, one has $\Tor_j^R(M,N)=0$ for all $j\geq n$. Tor-rigidity is quite a subtle topic that is not well understood in commutative algebra. More precisely, it is very difficult to examine properties of Tor-rigid modules, or determine if a given module is Tor-rigid. For example, it is a long-standing open problem whether modules that have finite projective dimension are Tor-rigid over  complete intersection rings. There are some recent results studying Tor-rigidity, but these mostly work over hypersurface domains; see, for example, \cite{Dsurvey}. As a consequence of Theorem \ref{t1}, we obtain information about  supports of Tor-rigid modules, which came as a surprise to us. Namely, we observe that the support of a maximal Cohen-Macaulay Tor-rigid module contains each height-two prime ideal over a local ring of dimension at least two; see Corollary \ref{corTor2}.

\begin{cor} \label{corTor1} Let $R$ be a Cohen-Macaulay local ring and let $M\in \md R$ be a nonzero module satisfying $(S_3)$. Assume there is a prime ideal $\fp$ with $\height(\fp)=2$ and $\fp\notin \Supp_R(M)$. Then $M$ is not Tor-rigid.  
\end{cor}

\begin{proof} Let $N=\Tr(R/\fp)$. Then, setting $n=1$, Theorem \ref{t1} shows $M\otimes_RN$ satisfies $(S_2)$, $\pd_R(N)=1$, $N$ is torsion-free but $N$ is not reflexive ($N$ is torsion-free since it satisfies $(S_1)$, but $N$ is not reflexive since it does not satisfy $(S_2)$ \cite[3.6]{EG}). As $\pd_R(N)<\infty$, $R$ is Cohen-Macaulay and $N$ is torsion-free, we have $N_{\fq}$ is free for each prime ideal $\fq$ of height at most one. In particular, $\Ext^1_R(\Tr N,M)$ is torsion.

Consider the four term exact sequence that follows from \cite{AuBr}:
\begin{equation} \tag{\ref{corTor1}.1}
0 \to \Ext^1_R(\Tr N,M)\rightarrow M\otimes_RN  \rightarrow \Hom_R(N^{\ast}, M)
	\rightarrow \Ext^2_R(\Tr N,M) \to 0.
\end{equation}
Since $M\otimes_RN$  is torsion-free and $\Ext^1_R(\Tr N,M)$ is torsion, it follows from (\ref{corTor1}.1) that $\Ext^1_R(\Tr N,M)=0$. We now follow an argument of Auslander: suppose $\Ext^2_R(\Tr N,M)\neq 0$ and pick $\fq$ in $\Ass_R(\Ext^2_R(\Tr N,M))$. Then $\height(\fq)\geq 2$, and $\fq \in \Supp_R(M\otimes_RN)$. Since $M$ satisfies $(S_3)$, it follows $\fq \notin \Ass_R(M)$. Hence $\fq\notin \Ass_R(\Hom_R(N^{\ast},M))$, and $\depth_{R}(\Hom_R(N^{\ast},M)_{\fq})\geq 1$. Thus, by localizing (\ref{corTor1}.1) at $\fq$ and using the depth lemma, we see that $\depth_{R_{\fq}}(M\otimes_R N)_{\fq}=1$; this is a contradiction since $M\otimes_RN$ satisfies $(S_2)$. So $\Ext^1_R(\Tr N,M)= \Ext^2_R(\Tr N,M)=0$. Consider the following exact sequence for $n=1, 2$ (see \cite{AuBr}):\\ 
$\Tor^R_{2}(\Tr\Omega^n\Tr N, M) \to \Ext^{n}_{R}(\Tr N,R)\otimes_{R}M \to \Ext^{n}_{R}(\Tr N, M) \to \Tor_{1}^{R}(\Tr\Omega^n\Tr N,M) \to 0$. 

Now, if $M$ is Tor-rigid, then $\Ext^1_R(\Tr N,R)=0= \Ext^2_R(\Tr N, R)$, i.e., $N$ is reflexive. Therefore $M$ cannot be a Tor-rigid module.
\end{proof}

\begin{cor} \label{corTor2} Let $R$ be a local ring of dimension at least two. If $M\in \md R$ is a nonzero maximal Cohen-Macaulay Tor-rigid module, then $\Supp_R(M)$ contains each height-two prime ideal of $R$.
\end{cor}

\begin{proof} A local ring admitting a maximal Cohen-Macaulay Tor-rigid module is Cohen-Macaulay; see \cite[4.3]{Au} and also \cite[4.7]{hom}. Hence the result follows from Corollary \ref{corTor1}.
\end{proof}

\section*{Acknowledgements}
The authors would like to thank Craig Huneke, Greg Piepmeyer, and Arash Sadeghi for valuable comments and suggestions, and Roger Wiegand for checking and simplifying our arguments. 

Celikbas is grateful to Craig Huneke, Greg Piepmeyer, Sean Sather-Wagstaff and Roger Wiegand for many useful discussions about the second rigidity theorem at different stages of this project.

\end{document}